\documentclass{amsart}

\theoremstyle{plain}
\newtheorem{theorem}{Theorem}
\newtheorem{corollary}{Corollary}
\newtheorem{lemma}{Lemma}

\theoremstyle{remark}
\newtheorem{remark}{Remark}

\numberwithin{equation}{section}

\allowdisplaybreaks[4]

\begin{document}

\title[Approximation by a special de la Vall\'ee Poussin\dots]{Approximation by a special de la Vall\'ee Poussin type matrix transform mean of Walsh-Fourier series} 

\author{Istv\'an Blahota}
\address{Institute of Mathematics and Computer Sciences\\
	University of Ny\'\i regyh\'aza\\
	H-4400 Ny\'\i regyh\'aza, S\'ost\'oi street 31/b\\
	Hungary}

\email{blahota.istvan@nye.hu}

\thanks{}

\begin{abstract}
	In this paper, we consider norm convergence for a special matrix-based de la Vall\'ee Poussin-like mean of Fourier series for the Walsh system. We estimate the difference between the named mean above and the corresponding function in norm, and the upper estimation is given by the modulus of continuity of the function.
\end{abstract}



\maketitle

\section{Definitions and notations}
We follow the standard notions of dyadic analysis introduced by F. Schipp, W. R. Wade, P. Simon, and J. P\'al \cite{SWSP} and others. 

\section{Definitions and notation}
Let $\mathbb P$ be the set of positive natural numbers and ${\mathbb N}:={\mathbb P}\cup \{0\}$.
Let denote by ${\mathbb Z}_2$ the discrete cyclic group of order
2, the group operation is the modulo $2$ addition. Let be every subset
open. The normalized Haar measure $\mu$ on ${\mathbb Z}_{2}$ is given
in the way that $\mu (\{ 0\}):=\mu (\{ 1\}):=1/2$. $G :=
\overset{\infty}{\underset{k=0}{{\times}}} {\mathbb Z}_{2},\ G$
is called the Walsh group. The elements of Walsh group $G$ are sequences of numbers 0 and 1, that is 
$x=(x_{0},x_{1},\dots,x_{k},\dots)$ with $x_{k}\in \{0,1\}\ (k\in {\mathbb
	N}).$

The group operation on $G$ is the coordinate-wise addition (denoted by $+$),
the normalized Haar measure $\mu $ is the product measure and the topology is the product topology.
Dyadic intervals are defined in the usual way 
\[
I_{0}(x):=G,\ I_{n}(x):=\{y\in G:y=(x_{0},\dots,x_{n-1},y_{n},y_{n+1},\dots)\}
\]
for $x\in G, n\in {\mathbb P}$. They form a base for the neighbourhoods of $G$.
Let $0:=(0:i\in
{\mathbb N})\in G$ denote the null element of $G$ and $I_{n}:=I_{n}(0)$ for $n\in {\mathbb N}$.

Let $L_{p}(G)$ denote the usual Lebesgue spaces on $G$ (with the
corresponding norm $\Vert .\Vert_p$), where $1\leq p<\infty$. 

For the sake of brevity in notation, we agree to write $L_{\infty}(G)$ instead of $C(G)$ and set $\Vert f\Vert_\infty:=\sup\{ \lvert f(x)\rvert : x\in G\}.$ Of course, it is clear that the space $L_{\infty}(G)$ is not the same as the space of continuous functions, i.e. it is a proper subspace of it. But since in the case of continuous functions the supremum norm and the $L_{\infty}(G)$ norm are the same, for convenience we hope the reader will be able to tolerate this simplification in notation.

Next, we define the modulus of continuity in $L_{p}(G), 1\leq p\leq\infty,$ of a function $f\in L_{p}(G)$ by
\[
\omega_p(f,\delta):=\sup_{\lvert t\rvert<\delta}\Vert f(.+t)-f(.)\Vert_p,\quad \delta>0,
\]
with the notation
\[
\lvert x\rvert:=\sum_{i=0}^\infty \frac{x_i}{2^{i+1}} \quad \textrm{for all }x\in G.
\]

The Lipschitz classes in $L_p(G)$ (for each $\alpha>0$) are defined as
\[
\textrm{Lip}(\alpha,p,G):=\{ f\in L_{p}(G): \omega_p(f,\delta)=O(\delta^\alpha )
\textrm{ as } \delta\to 0\}.
\]

We introduce some concepts of Walsh-Fourier analysis.
The Rade\-macher functions are
defined as
\[
r_{k}(x):=(-1)^{x_{k}} \ (x\in G, k\in {\mathbb N}).
\]
The Walsh-Paley functions are the product functions of the Rademacher functions. Namely,
each natural number $n$ can be uniquely expressed in the number system based 2, in the form 
\[
n=\sum_{k=0}^\infty n_{k}2^{k},\ n_{k}\in \{ 0,1\} \ (k\in {\mathbb N}),
\] 
where only a finite number of $n_{k}$'s different from zero.
Let the order of $n\in\mathbb{P}$ be denoted by $\vert n\vert :=\max \{ j\in
{\mathbb N}: n_{j}\neq 0\}$. Walsh-Paley functions are $w_{0}:=1$ and for $n\in\mathbb{P}$
\[
w_{n}(x):=\prod_{k=0}^{\infty }r_{k}^{n_{k}}(x)
=(-1)^{\sum_{k=0}^{\vert n\vert }n_{k}x_{k}}.
\]

Let ${\mathcal P}_n$ be the collection of Walsh  polynomials of order less than $n$, that is, functions of the form
\[
P(x)=\sum_{k=0}^{n-1}a_kw_k(x),
\]
where $n\in\mathbb{P}$ and $\{a_k\}$ is a sequence of complex numbers.

It is known \cite{hew} that the system $(w_{n}, n\in
{\mathbb N})$ is the character system of $(G,+)$.
The $n$th Fourier-coefficient, the $n$th partial sum of the Fourier series and the $n$th Dirichlet kernel is defined by
\[
\hat f(n):=
\int_{G} fw_{n} d\mu,
\
S_{n}(f):=\sum_{k=0}^{n-1}\hat f(k)w_{k},
\ D_{n}:=
\sum_{k=0}^{n-1}w_{k},\ D_{0}:=0. 
\]
Fej\'er kernels are defined as the arithmetical means of Dirichlet kernels, that is, 
\[
K_{n}:=\frac{1}{n}\sum_{k=1}^{n}D_{k}.
\] 

Let $T:=\left( t_{i,j}\right)_{i,j=1}^{\infty}$ be a doubly infinite matrix of numbers. It is always supposed that matrix $T$ is lower triangular. 

Let us define the $(m,n)$th  matrix transform de La Vall\'ee Poussin mean  determined by the matrix $T$ as
\[
	\sigma_{m,n}^T(f):=\sum_{k=m}^{n}t_{k,n}S_k(f),
\]
where $m,n\in\mathbb{P}$ and $m\leq n$.

The $(m,n)$th matrix transform de La Vall\'ee Poussin kernel is defined as
\[
	K_{m,n}^{T}:=\sum_{k=m}^{n}t_{k,n}D_{k}.
\]
It is very easy to verify that
\[
	\sigma_{m,n}^T(f;x)=\int_{G}f(u)K_{m,n}^{T}(u+x)d\mu(u).
\]

We introduce the notation
$\Delta t_{k,n}:=t_{k,n}-t_{k+1,n},$ where $k\in\{1,\ldots,n\}$ and $t_{n+1,n}:=0$.	

\section{Historical overview}

Matrix transforms means are common generalizations of several well-known summation methods. It follows by simple consideration that the N\"orlund means, the  Fej\'er (or the $(C,1)$) and the $(C,\alpha)$ means are special cases of the matrix transform summation method introduced above.

Our paper is motivated by the work of M\'oricz, Siddiqi \cite{MS} on the Walsh-N\"orlund summation method and the result of M\'oricz and Rhoades \cite{MR} on the Walsh weighted mean method. As special cases, M\'oricz and Siddiqi obtained the earlier results given by
Yano \cite{Y2}, Jastrebova \cite{J} and Skvortsov \cite{Sk2} on the rate of the approximation by Ces\`aro means. The approximation properties of the Walsh-Ces\`aro means of negative order were studied by Goginava \cite{gog1}, the Vilenkin case was investigated by Shavardenidze \cite{Sh} and Tepnadze \cite{TT}. A common generalization of these two results of M\'oricz and Siddiqi \cite{MS} and M\'oricz
and Rhoades \cite{MR} was given by Nagy and the author \cite{BN2}.

In 2008, Fridli, Manchanda and Siddiqi  generalized the result of M\'oricz and Siddiqi for homogeneous Banach spaces and dyadic Hardy spaces \cite{FSM}. Recently, the author, Baramidze, Memi\'c, Nagy, Persson, Tephnadze and Wall presented  some results with respect to this topic \cite{T1},\cite{BN5}, \cite{T3},\cite{T2}. See \cite{TE, W1}, as well. For the two-dimensional results see \cite{BNT,N1,N2}.

It is important to note that in the paper of Chripk\'o \cite{C1} some methods and results with respect to Jacobi-Fourier series gave us some ideas and used in this paper.

\section{Auxiliary results}	

To prove Theorem \ref{thm-main} we need the following Lemmas. 
\begin{lemma}[Paley's Lemma \cite{SWSP}, p. 7.]\label{pal} For $n\in\mathbb{N}$
	\[
	D_{2^n}(x) = \begin{cases} 2^n, &\textrm{ if } x\in I_n, \\
		0, &\textrm{ if } x\notin I_n.
	\end{cases}
	\]
\end{lemma}

\begin{lemma}[\cite{SWSP}, p. 34.]\label{l3}
	For $j,n\in\mathbb{N},\ j<2^n$ we have
	\[
	D_{2^n+j}=D_{2^n}+r_{n}D_j.
	\]
\end{lemma}

\begin{lemma}[Yano's Lemma \cite{Y1}]\label{Yano}  The norm of the Fej\'er kernel is bounded uniformly. That is, for all $n\in\mathbb{P}$
	\[
	\|K_{n}\|_{1}\leq 2.
	\]
\end{lemma}
In 2018, Toledo improved this result.
\begin{lemma}\cite{Tol}\label{Tol}  
	\[
	\sup_{n\in\mathbb{P}}\|K_{n}\|_{1}= \frac{17}{15}.
	\]
\end{lemma}

\begin{lemma}\cite{MS}\label{MS-mod} Let $n\in\mathbb{P},\ g\in {\mathcal P}_{2^{n}}$, $f\in L_p(G)$ ($1\leq p\leq \infty$). Then 
	\[
	\left\Vert\int_{G} r_{n}(t) g(t) (f(\cdot+t)-f(\cdot))d\mu(t) \right\Vert_p\leq \frac{1}{2}\Vert g\Vert_1\omega_p\left(f,2^{-n}\right)
	\]
	holds.
\end{lemma}

In the next lemma, we give a decomposition of the kernels $K_{2^{n},2^{n+1}-1}^T$.
\begin{lemma}\label{decomp} Let $n$ be a positive integer, then we have 
	\begin{align*}
		K_{2^{n},2^{n+1}-1}^{T}
		&=\sum_{k=0}^{2^{n}-1}t_{2^{n} +k,2^{n+1}-1}D_{2^{n}}+r_{n}\sum_{k=1}^{2^{n}-2}\Delta t_{2^{n} +k,2^{n+1}-1}kK_k
		\\
		&\quad 
		+r_{n}t_{2^{n+1}-1,2^{n+1}-1}(2^{n}-1)K_{2^{n}-1}\\
		&=:\sum_{j=1}^{3} K_{j,n}.
	\end{align*}
\end{lemma}
\begin{proof}
	We write 
	\begin{align*}
		K_{2^{n},2^{n+1}-1}^{T}=\sum_{l=2^{n}}^{2^{n+1}-1}t_{l,2^{n+1}-1}D_{l}.
	\end{align*}
	
	Now, we apply Lemma \ref{l3}. We get 
	\begin{align*}
		\sum_{l=2^{n}}^{2^{n+1}-1}t_{l,2^{n+1}-1}D_{l}&=\sum_{k=0}^{2^{n}-1}t_{2^{n} +k,2^{n+1}-1}D_{2^{n}+k}\\
		&=\sum_{k=0}^{2^{n}-1}t_{2^{n} +k,2^{n+1}-1}D_{2^{n}}
		+r_{n}\sum_{k=1}^{2^{n}-1}t_{2^{n} +k,2^{n+1}-1}D_{k}.
	\end{align*}
	Using Abel-transform 
	\begin{align*}\sum_{k=1}^{2^{n}-1}t_{2^{n} +k,2^{n+1}-1}D_{k}&=\sum_{k=1}^{2^{n}-2}\Delta  t_{2^{n} +k,2^{n+1}-1}kK_k\\
		&\quad+\ t_{2^{n+1}-1,2^{n+1}-1}(2^{n}-1)K_{2^{n}-1}.
	\end{align*}		
	
	Summarizing these it completes the proof of Lemma \ref{decomp}.
\end{proof}

\section{The rate of the approximation}

\begin{theorem}\label{thm-main} 
	Let $f\in L_p(G)\ (1 \leq p \leq \infty)$. For every $n\in\mathbb{P},\ \{t_{k,2^{n+1}-1}: 2^{n}\leq k\leq 2^{n+1}-1\}$ be a finite sequence of non-negative numbers such that 
	\begin{equation}\label{c1}
		\sum_{k=2^{n}}^{2^{n+1}-1}t_{k,2^{n+1}-1}=1
	\end{equation}
	be satisfied.
	
	a) If the finite sequence  $ \{ t_{k,2^{n+1}-1}: 2^{n}\leq k \leq 2^{n+1}-1\}$ is non-decreasing for a fixed $n$	and 
	\begin{equation}\label{c2}
		t_{2^{n+1}-1,2^{n+1}-1}=O\left(\frac{1}{2^{n+1}-1}\right),
	\end{equation} or 
	
	b) if the finite sequence $ \{ t_{k,2^{n+1}-1}: 2^{n}\leq k \leq 2^{n+1}-1\}$ is  non-increasing for a  fixed $n$, then
	\[
		\left\Vert\sigma_{2^{n},2^{n+1}-1}^T(f)-f\right\Vert_{p}\leq
		c\omega_{p}\left(f,2^{-n}\right)
	\]
	holds.
\end{theorem}
\begin{proof}[Proof of Theorem \ref{thm-main}] 
	The proof is carried out in cases where  $1\leq p<\infty$, while the proof of case  $ p=\infty$ is similar. Recall that by the case $p=\infty$ we mean that we are considering the space of continuous functions.
	
	During our proofs $c$ denotes a positive constant, which may vary at different appearances. 
	
	We use  condition \eqref{c1}, the usual Minkowski inequality and Lemma \ref{decomp}
	\begin{align*}
		\left\Vert\sigma_{2^{n},2^{n+1}-1}^T(f)-f\right\Vert_{p}&=\left(\int_{G}|\sigma_{2^{n},2^{n+1}-1}^T(f;x)-f(x)|^{p}
		d\mu(x)\right)^{\frac{1}{p}}\\
		&=\left(\int_{G}\left|\int_{G}K_{2^{n},2^{n+1}-1}^{T}(u)F(x,u)d\mu(u)\right|^{p}
		d\mu(x)\right)^{\frac{1}{p}}\\
		&\leq\sum_{j=1}^{3}\left(\int_{G}\left|\int_{G}K_{j,n}(u)F(x,u)d\mu(u)\right|^{p}
		d\mu(x)\right)^{\frac{1}{p}}	\\	
		&=:\sum_{j=1}^{3}I_{j,n}
	\end{align*}
	with notation $F(x,u):=f(x+u)-f(x)$.
	
	Using generalized Minkowski inequality,  Lemma \ref{pal} and  condition \eqref{c1} for the expressions $I_{1,n}$, 
	we obtain
	\begin{align*}
		I_{1,n}&\leq 
		\sum_{k=0}^{2^{n}-1} t_{2^{n}+k,2^{n+1}-1} 		 
		\int_{G}D_{2^{n}}(u)\left(\int_{G}\left|F(x,u)\right|^{p}d\mu(x)\right)^{\frac{1}{p}}d\mu(u)\nonumber\\
		&\leq \sum_{k=2^{n}}^{2^{n+1}-1} t_{k,2^{n+1}-1} 
		\omega_{p}\left(f,2^{-n}\right)\\
		&\leq
		\omega_{p}\left(f,2^{-n}\right).
	\end{align*}
	
	Now, applying Lemma \ref{Tol} and Lemma \ref{MS-mod} we get 
	\begin{align*}
		I_{2,n}
		&\leq\sum_{k=1}^{2^{n}-2}\left|\Delta t_{2^{n}+k,2^{n+1}-1}\right|k\\
		&\quad\times\left(\int_{G}\left|\int_{G}r_{n}(u)K_{k}(u)F(x,u)d\mu(u)\right|^{p}d\mu(x)\right)^{\frac{1}{p}}\\
		&\leq
		\sum_{k=1}^{2^{n}-2}\left|\Delta t_{2^{n}+k,2^{n+1}-1}\right|k \frac{1}{2}\|K_{k}\|_{1}\omega_{p}\left(f,2^{-n}\right)\\
		&\leq
		\sum_{k=1}^{2^{n}-2}\left|\Delta t_{2^{n}+k,2^{n+1}-1}\right|k \omega_{p}\left(f,2^{-n}\right).
	\end{align*}
	
	We write in case a) 
	\begin{align*}
		\sum_{k=1}^{2^{n}-2} |\Delta t_{2^{n}+k,2^{n+1}-1}|k
		&=\sum_{k=1}^{2^{n}-2}
		(t_{2^{n}+k+1,2^{n+1}-1}-t_{2^{n}+k,2^{n+1}-1} )k\nonumber\\
		&= (2^{n}-2)t_{2^{n+1}-1,2^{n+1}-1}-\sum_{k=1}^{2^{n}-2}t_{2^{n}+k,2^{n+1}-1}\\
		&\leq  
		(2^{n+1}-1) t_{2^{n+1}-1,2^{n+1}-1}\nonumber
	\end{align*}
	and using condition \eqref{c2}
	\begin{align*}
		I_{2,n}&\leq(2^{n+1}-1) t_{2^{n+1}-1,2^{n+1}-1}\omega_{p}\left(f,2^{-n}\right)\\
		&\leq c\omega_{p}\left(f,2^{-n}\right).
	\end{align*}
	
	We estimate the expression $I_{3,n}$ in case a). Lemma \ref{Tol}, Lemma \ref{MS-mod} and condition \eqref{c2} yield 
	\begin{align*}
		I_{3,n}&\leq
		(2^{n}-1)t_{2^{n+1}-1,2^{n+1}-1}\\ &\quad \times\left(\int_{G}\left|\int_{G}r_{n}(u)K_{2^{n}-1}(u)F(x,u)d\mu(u)\right|^{p}d\mu(x)\right)^{\frac{1}{p}}\\
		&\leq
		(2^{n+1}-1)t_{2^{n+1}-1,2^{n+1}-1} \frac{1}{2}
		\|K_{2^{n}-1}\|_{1}\omega_{p}\left(f,2^{-n}\right)\nonumber\\
		&\leq
		(2^{n+1}-1)t_{2^{n+1}-1,2^{n+1}-1} 
		\omega_{p}\left(f,2^{-n}\right)\\ &\leq
		c\omega_{p}\left(f,2^{-n}\right).
	\end{align*}

	In case b) we estimate $I_{2,n}+I_{3,n}$. In this situation
	\begin{align*}
		\sum_{k=1}^{2^{n}-2}\left|\Delta t_{2^{n}+k,2^{n+1}-1}\right|k
		&= \sum_{k=1}^{2^{n}-2}t_{2^{n}+k,2^{n+1}-1}-(2^{n}-2)t_{2^{n+1}-1,2^{n+1}-1},
	\end{align*}
	so Lemma \ref{Tol}, Lemma \ref{MS-mod} and condition \eqref{c1} imply
	\begin{align*}
	I_{2,n}+I_{3,n}&\leq
	\sum_{k=1}^{2^{n}-2}\left|\Delta t_{2^{n}+k,2^{n+1}-1}\right|k\times\\ &\quad\quad \times\left(\int_{G}\left|\int_{G}r_{n}(u)K_{k}(u)F(x,u)d\mu(u)\right|^{p}d\mu(x)\right)^{\frac{1}{p}}\\
	&\quad+
	(2^{n}-1)t_{2^{n+1}-1,2^{n+1}-1}\times\\ &\quad\quad \times\left(\int_{G}\left|\int_{G}r_{n}(u)K_{2^{n}-1}(u)F(x,u)d\mu(u)\right|^{p}d\mu(x)\right)^{\frac{1}{p}}\\
	&\leq
	\left(\sum_{k=1}^{2^{n}-2} t_{2^{n}+k,2^{n+1}-1}+(2^{n}-1)t_{2^{n+1}-1,2^{n+1}-1}-(2^{n}-2)t_{2^{n+1}-1,2^{n+1}-1}\right)\\
	&\quad\times\frac{1}{2}\cdot\frac{17}{15}\omega_{p}\left(f,2^{-n}\right)\\
	&=
	\sum_{k=1}^{2^{n}-1}t_{2^{n}+k,2^{n+1}-1} \frac{17}{30}\omega_{p}\left(f,2^{-n}\right)
	\\ &\leq
	\frac{17}{30}\omega_{p}\left(f,2^{-n}\right).
\end{align*}	
	
	This completes the proof of our Theorem \ref{thm-main}.
\end{proof}
\begin{remark}
	We mention, that assuming \eqref{c1} is natural, because many well-known means satisfy it and this equality is a part of regularity conditions \cite[page 74.]{Zyg}.
\end{remark}
\begin{corollary}\label{s}
	Let us suppose that the conditions in Theorem \ref{thm-main} are satisfied. If $f\in \textrm{Lip}(\alpha,p,G)$, then   
	\[
	\Vert\sigma_{2^{n},2^{n+1}-1}(f)-f\Vert_{p}=O\left(2^{-n\alpha}\right).
	\] 
\end{corollary}

\begin{remark}
In case b) we can formulate the statement of Theorem \ref{thm-main} in following form
\[
	\left\Vert\sigma_{2^{n},2^{n+1}-1}^T(f)-f\right\Vert_{p}\leq
	\frac{47}{30}\omega_{p}\left(f,2^{-n}\right).
\]
\end{remark}

\bibliography{B_dLVP2n_2023.06.13_arxiv}
\bibliographystyle{plain}
\end{document}